\documentclass{amsart}[11pt]
\usepackage{amsmath,amsthm,amsfonts,amssymb,amscd,enumerate,verbatim}
\usepackage{varwidth,stmaryrd}
\usepackage{graphicx}

\usepackage[all]{xy}\SelectTips{eu}{}

\usepackage{color}

\newcommand{\n}{{\mathfrak n}}
\newcommand{\m}{{\mathfrak m}}

\newcommand{\Image}{\operatorname{Image}}

\newcommand{\coker}{\operatorname{coker}}

\newcommand{\rank}{\operatorname{rank}}

\newcommand{\Tor}{\operatorname{Tor}}

\newcommand{\Cone}{\operatorname{Cone}}

\newcommand{\Hom}{\operatorname{Hom}}

\newcommand{\V}{\operatorname{V}}
\newcommand{\Z}{\operatorname{Z}}

\newcommand{\Ktac}{\operatorname{\mathsf{K}_{tac}}\nolimits}
\newcommand{\grKtac}{\operatorname{\mathsf{grK}_{tac}}\nolimits}
\newcommand{\Id}{\operatorname{Id}\nolimits}

\renewcommand{\H}{\operatorname{H}\nolimits}

\numberwithin{equation}{section}

\theoremstyle{plain}
\newtheorem{theorem}{Theorem}[section]

\newtheorem*{Main Theorem}{Main Theorem}

\newtheorem{proposition}[theorem]{Proposition}
\newtheorem{lemma}[theorem]{Lemma}
\newtheorem{corollary}[theorem]{Corollary}

{\Alph{theorem}}
{\Alph{theorem}}

\theoremstyle{definition}

\newtheorem{chunk}[theorem]{}
\newtheorem{remark}[theorem]{Remark}
\newtheorem{definition}[theorem]{Definition}
\newtheorem{example}[theorem]{Example}

  \newcounter{numlist} %
  {\end{list}}%

\theoremstyle{remark}

\numberwithin{equation}{theorem}

\begin{document}

\title[Rank Varieties over the Generic Hypersurface I]{Rank Varieties over the Generic Hypersurface I}
\author{David A. Jorgensen}
\address[D. A. Jorgensen]{Department of Mathematics, University of Texas at Arlington, 411 S. Nedderman Drive, Pickard Hall 429, Arlington, TX 76019, USA}
\email{djorgens@uta.edu}

\date{\today}  

\begin{abstract} To every local complete intersection ring one may associate a so-called generic hypersurface. In this paper we introduce rank varieties for modules and complexes over the generic hypersurface. The definition uses extension of scalars, rather than restriction of scalars which are used to define the conventional support varieties over a local complete intersection. We show that every projective variety can be realized as the rank variety of a finitely generated module over the generic hypersurface. We also investigate several properties of these rank varieties.
\end{abstract}

\maketitle

\section{Introduction}\label{sec:1}

There has been significant interest in rank and support varieties for cohomology of modules or complexes over various rings --- for example, commutative local complete intersection rings --- no doubt due to the elegant and useful connection between homological module theory and geometry (see \cite{BenIKP},\cite{BerPW1},\cite{BerPW2}, \cite{FP}, \cite{NP},\cite{PW}, for example, and \cite{C2} for a survey).  Building on Quillen's geometric treatment of cohomology of finite groups \cite{Q}, rank and support varieties were introduced by Carlson in \cite{C1} for cohomology of modules over modular group algebras of elementary abelian $p$-groups. He and Avrunin and Scott \cite{AvS} showed that these two notions of variety, although defined quite differently, agree.  Avramov in \cite{A1} defines support varieties now for cohomology of modules over commutative local complete intersection rings (of which the aforementioned rings studied by Carlson are a special case) and Avramov and Buchweitz in \cite{AB} extend the notion of support varieties to cohomology of pairs of modules over the same type of rings.  In that paper they also define rank varieties (without calling them so) and prove they coincide with the support varieties \cite[Theorem 2.5]{AB}. Steele defines rank and support varieties for cohomology of totally acyclic complexes in \cite{S} and shows that these varieties also coincide.
There are also now quite general and sophisticated notions of support, see \cite{BenIK} for example. 
Although equivalent when both defined, an argument can be made that rank varieties are more concrete and easier to compute than support varieties. 

To every local complete intersection ring one may associate a so-called generic hypersurface ring. Generic hypersurfaces were used in \cite[Section 3]{AB} and \cite{J} to study support varieties over the associated complete intersection. The point of this paper is to introduce and study a notion of rank variety over the generic hypersurface itself. The reason for this is the last sentence of the previous paragraph and the remarkable result of Orlov \cite[Theorem 2.1]{O1} that proves the singularity category of the local complete intersection is equivalent to that of the generic hypersurface. Buchweitz \cite{B} shows that singularity category is equivalent to the category of totally acyclic complexes (straightforward modifications yield the equivalence in the graded setting).  Therefore we focus our study on rank varieties of graded totally acyclic complexes over the generic hypersurface. In a subsequent paper we study the correlation between the conventional support varieties of totally acyclic complexes over a local complete intersection and rank varieties of graded totally acyclic complexes over the generic hypersurface via these equivalences. All previous treatments of rank varieties involve restriction of scalars. Our approach instead uses extension of scalars. It turns out this framework leads to much simpler proofs of closure and an affirmative answer to the realizability question: is every variety the rank variety of some complex or module? Another related area of interest is matrix factorizations (see \cite{BW1},\cite{BW2} and \cite{O3}, for example).  Buchweitz \cite{B} and Orlov \cite{O2} show that the triangulated category of matrix factorizations of an element $w$ is equivalent to the category of totally acyclic complexes over the hypersurface defined by $w$. Thus one can interpret the rank varieties introduced here as those of graded matrix factorizations as well. Connections between the category of totally acycic complexes over the generic hypersurface and those over specialized hypersurfaces were established in \cite{J2}.  

We next describe briefly the framework of rank varieties for commutative local complete intersections.
Recall that a local complete intersection ring 
$S$ is the quotient of a commutative regular local ring $Q$ by an ideal generated by a regular sequence $f_1,\dots,f_c$. Without loss of generality, one may assume that each $f_i$ lies in the square of the unique maximal ideal $\n$ of $Q$.  In this case we say that $S$ has codimension $c$ and if $k=Q/\n$ denotes the residue field of $Q$, then one can associate to $S$, or, rather $I=(f_1,\dots,f_c)$, the projective space $\mathbb P^{c-1}_k$ of dimension 
$c-1$ according to 
\[
a_1f_1+\cdots +a_cf_c\in I \leftrightarrow \alpha=(\alpha_1,\dots,\alpha_c)\in\mathbb P^{c-1}_k
\]
where $a_i\in Q$ is a preimage of $\alpha_i\in k$ for each $i=1,\dots,c$. Note that at least one 
$a_i$ must be a unit in $Q$, so that $a_1f_1+\cdots +a_cf_c$ is always a minimal generator of $I$.  In this case we call $Q/(a_1f_1+\cdots +a_cf_c)$ a permissible hypersurface. Also note that distinct elements of $I$ may correspond to the same element in $\mathbb P^{c-1}_k$. Specifically, $a_1f_1+\cdots +a_cf_c$ and $a'_1f_1+\cdots +a'_cf_c$ correspond to the same element 
$\alpha\in\mathbb P^{c-1}_k$ if and only if $a_i-a'_i\in\n$ for $i=1,\dots,c$. As we shall see, this nonunique correspondence causes no problems in the theory.

The previous considerations of rank varieties use certain ring maps $R_\alpha\to S$, and parameterize by $\alpha\in\mathbb P^{c-1}_k$ those $R_\alpha$ over which a given $S$-module or complex has infinite projective dimension or is nonzero. Thus these notions of rank variety are defined through restriction of scalars via the maps $R_\alpha\to S$. In the group algebra case the rings $R_\alpha$ are principle subalgebras of $S$, with the maps being inclusion. In the local complete intersection case the rings $R_\alpha$ are the permissible hypersurfaces $Q/(a_1f_1+\cdots+a_cf_c)$ and the maps 
$R_\alpha\to S$ are the natural surjections.

In this paper we instead consider a rank variety defined through extension of scalars via the natural
maps
\[
R \to R_\alpha 
\]
where $R$ is the generic hypersurface 
\[
R=Q[x_1,\dots,x_c]/(f_1x_1+\cdots+f_cx_c)
\] 
and again parameterize by $\alpha\in\mathbb P^{c-1}_k$ those permissible hypersurface rings 
$R_\alpha=Q/(a_1f_1+\cdots+a_cf_c)$ over which a given $R$-module or complex has infinite projective dimension or is nonzero after tensoring with $R_\alpha$.  

In Section 3 we define this version of rank variety for graded totally acyclic complexes over the generic hypersurface $R$. We show in 
\ref{independent} below that our notion of rank variety is well-defined, that is, it is independent of the choice of preimages $a_i$ of $\alpha_i$.  We also show it is well-defined up to homotopy equivalence, 
\ref{homeq}. We then show it is a Zariski closed set \ref{closed} and give in \ref{properties} several properties of rank varieties analogous to those defined via restriction of scalars.  We show in Section 4 that every projective variety in $\mathbb P^{c-1}_k$ is the rank variety of some graded totally acyclic $R$-complex and that a sort of Dade's Lemma \ref{Dades} holds in our context. Contrarily, Carlson's connectedness theorem does not hold in our context: we show by example in Section 5 that disconnectedness of the rank variety does not necessarilly imply decomposability of the object.

In Section \ref{modules} we restrict out attention graded modules over the generic hypersurface $R$ and establish the same properties for rank varieties of modules. Section 2 consists of preliminary results on complexes in general and graded totally acyclic complexes in particular.

\section{Preliminaries}\label{sec:2}

In this section we assume that $R$ is a commutative Noetherian ring. 

\subsection*{Complexes}

Recall that a (chain) $R$-complex $C$ is a sequence of $R$-module homomorphisms
\[
C: \cdots \to C_{n+1} \xrightarrow{\partial_{n+1}^C} C_n \xrightarrow{\partial_n^C} C_{n-1}\to\cdots
\]
such that $\partial_n^C\partial^C_{n+1}=0$ for all $n\in\mathbb Z$. A (chain) map $f:C\to D$ of $R$-complexes
is a family of maps $f_n:C_n\to D_n$ such that $\partial_n^Df_n=f_{n-1}\partial^C_n$ for all $n\in\mathbb Z$,
in other words, the squares in the following diagram commute.
\[
\xymatrixrowsep{2pc}
\xymatrixcolsep{3pc}
\xymatrix{
\cdots \ar@{->}[r] & C_{n+1} \ar@{->}[r]^{\partial_{n+1}^C}\ar@{->}[d]^{f_{n+1}} & C_n \ar@{->}[r]^{\partial_{n}^C}\ar@{->}[d]^{f_{n}} & C_{n-1} \ar@{->}[r]\ar@{->}[d]^{f_{n-1}} & \cdots \\
 \cdots \ar@{->}[r] & D_{n+1} \ar@{->}[r]^{\partial_{n+1}^D} & D_n \ar@{->}[r]^{\partial_{n}^D} & D_{n-1} \ar@{->}[r] & \cdots 
}
\] 

For an $R$-complex $C$ we have the shift operation $\Sigma C$, which returns another $R$-complex with
$(\Sigma C)_n=C_{n-1}$ and $\partial_n^{\Sigma C}=-\partial_{n-1}^C$. 

Given a map $f:C\to D$ of $R$-complexes, one constructs the mapping cone of $f$, $\Cone(f)$, which is another
$R$-complex defined by $\Cone(f)_n=D_n\oplus \Sigma C_n$ and 
$\partial_n^{\Cone(f)}$ is represented by the matrix 
$\left(\begin{array}{cc} \partial_n^D & f_{n-1}\\ 0 & -\partial_{n-1}^C \end{array}\right)$

\begin{chunk}\label{compose null}
Recall that a chain map $f:C\to D$ is null-homotopic if there exist maps $s_n:C_n\to D_{n+1}$ such that $\partial_{n+1}^Ds_n+s_{-1}\partial_n^C=f_n$ for all $n$; in this case we write $f\sim 0$. Two chain maps $f,g:C\to D$ are homotopic if $f-g\sim 0$, and this case we write $f\sim g$. It is easy to see that if $f\sim 0$ then for any compositions, $fg\sim 0$ and $hf\sim 0$. 
\end{chunk}

\subsection*{Contractibility} In this subsection we give the definition of contractibility and some properties we use later.

\begin{definition} A complex $C$ of finitely generated free $R$-modules is \emph{contractible at $n\in\mathbb Z$} if there exist maps $s_n:C_n\to C_{n+1}$ and $s_{n-1}:C_{n-1}\to C_n$ such that 
$\partial^C_{n+1}s_n+s_{n-1}\partial^C_n$ is an isomorphism.  In this case we say that $s_n$ and $s_{n-1}$ is a \emph{contraction} at $n$. The complex $C$ is \emph{contractible} if there exist maps 
$s_n:C_n\to C_{n+1}$ for each $n\in\mathbb Z$ such that $s_n$ and $s_{n-1}$ is a contraction at $n$ for all 
$n\in\mathbb Z$.
\end{definition}

\begin{remark} Typically one defines $C$ to be contractible if the identity map $\Id_C$ is 
nullhomotopic, but this is equivalent to our definition. Indeed, if 
$f_n=\partial^C_{n+1}s_n+s_{n-1}\partial^C_n$ is an isomorphism for all $n$, then the maps 
$s_nf_n^{-1}$ form a homotopy showing that $\Id_C$ is nullhomotopic.
\end{remark}

\begin{chunk}\label{contract equiv} Recall that complexes $C$ and $D$ are homotopically equivalent if there exist chain maps
$f:C\to D$ and $g:D\to C$ such that $gf\sim 1_C$ and $fg\sim 1_D$. If $C$ and $D$ are homotopically equivalent and $C$ is contractible then $D$ is also contractible.  Indeed, by \ref{compose null} 
$1_C\sim 0$ implies $f\sim 0$ which in turn implies $1_D\sim fg\sim 0$.
\end{chunk}

\subsection*{Vector spaces}

Recall that we say a map $f$ between finite dimensional vector spaces has \emph{rank $n$}, denoted 
$\rank f=n$, if the the vector space dimension of $\Image f$ is $n$.  Equivalently, $\rank f$ is the largest size nonzero minor in a matrix representing $f$; this size is independent of the matrix representing $f$. For consistency, we denote the vector space dimension of a finite dimensional vector space $V$ by $\rank V$.
The rank-nullity Theorem from linear algebra gives the following.

\begin{chunk}\label{rank=} Consider a complex $C$ of vector spaces. The following are equivalent: 
\begin{enumerate}
\item $C$ is exact at $n$;
\item $\rank\partial^C_{n+1}+\rank\partial^C_n=\rank C_n$;
\item $C$ is contractible at $n$.
\end{enumerate}
\end{chunk}

\subsection*{Integral domains}
Assume that $R$ is an integral domain.  Then $S=R-\{0\}$ is a multiplicatively closed set and each map of finitely generated $R$-modules $f:M\to N$ has a well-defined rank, which we denote by 
$\rank f$; it is the rank of $S^{-1}f$. If $M$ and $N$ are moreover free modules then $f$ can be represented by a matrix $A$ with respect to chosen bases of 
$M$ and $N$. For an integer $r$ we let $I_r(A)$ denote the ideal of $R$ generated by the 
$r\times r$ minors of $A$.  It is well-known that if $B$ is another matrix representing
$f$ with respect to some other choice of bases of $M$ and $N$ then 
$I_r(B)=I_r(A)$.  We write $I_r(f)$ for this common ideal. A key fact then is the following

\begin{chunk}\label{maprank}
For a map $f:M\to N$ of finitely generated free modules we have
\[
\rank f=\sup\{r\mid I_r(f)\ne 0\}
\]
Indeed, if $A$ is a matrix representing $f$ then $I_r(f)=0$ if and only if $S^{-1}I_r(f)=0$ if and only if $S^{-1}I_r(A)=I_r(S^{-1}A)=0$, and by definition, 
$\rank S^{-1}f=\sup\{r\mid I_r(S^{-1}A)\ne 0\}$.
\end{chunk}

\subsection*{Local rings and minimality} We recall the definition of minimal complex and state some properties.

\begin{chunk}\label{surjective1} Assume that $(R,\m,k)$ is a local ring.
\begin{enumerate}
\item  Let $f:M\to N$ a homomorphism of 
finitely generated $R$-modules.  Then $f$ is surjective if and only if the induced map $\overline f:M/\m M\to N/\m N$ is surjective. The `only if' implication is straightforward.  The `if' direction uses Nakayama's Lemma.
\item Let $f:F\to F$ be an endomorphism of a finitely generated free $R$-module.  Then $f$ is an isomorphism if and only if $f$ is surjective.
\end{enumerate}
\end{chunk}

Assume that $C$ is a complex of finitely generated free $R$-modules and that $(R,\m,k)$ is a local ring.
Recall that we say $C$ is \emph{minimal} if $\partial^C_n(C_n)\subseteq\m C_{n-1}$ for all 
$n\in\mathbb Z$.

We now show that every $R$-complex of finitely generated free modules contains a minimal complex as a direct summand.

\begin{proposition}\label{directsum} Assume that $(R,\m,k)$ is a local ring and let $C$ be a complex of finitely generated free $R$-modules.  Then $C$ is isomorphic to the direct sum of a contractible complex $E$ and a minimal complex $X$.
\end{proposition}

\begin{proof} For each $n\in\mathbb Z$ choose a subset $B_{n+1}\subseteq C_{n+1}$ such that 
\[
\{\partial^C_{n+1}(b)+\m C_n\mid b\in B_{n+1}\}
\] 
is a $k$-basis of $\partial^C(C_{n+1})+\m C_n/\m C_n$. Note that since $\partial_{n+1}^C(B_{n+1})\cup B_n$ is 
linearly independent modulo $\m C_n$, by Nakayama's Lemma, it extends to a basis of $C_n$. Define
$E_n=\sum_{b\in\partial_{n+1}^C(B_{n+1})\cup B_n}Rb$.  Then clearly
\[
E: \cdots \to E_{n+1}\xrightarrow{\partial^C_{n+1}|E_{n+1}} E_n \xrightarrow{\partial^C_n|E_n} E_{n-1} \to \cdots
\]
Is a contractible subcomplex of $C$. Let $X$ be the quotient complex $C/E$, which is minimal by construction. It is an easy exercise to see that $C$ and $E\oplus X$ are isomorphic as complexes.
\end{proof}

\subsection*{Totally acyclic complexes}
We let $\Ktac(R)$ denote the subcategory of the homotopy category of chain complexes over $R$ consisting of the \emph{totally acyclic complexes}. That is, $\Ktac(R)$ is the category whose objects are the chain complexes $C$ of finitely generated free $R$-modules with
\[
\H(C)=0=\H(\Hom_R(C,R))
\]
and the morphisms $g:C\to D$ are the homotopy equivalence classes of chain maps.
We will often abuse terminology and call morphisms in $\Ktac(R)$ chain maps.


\begin{chunk}\label{CR}
Recall (see, for example, \cite{AM} or \cite{B}) that a complete resolution of an $R$-module $M$ is an object $C\in\Ktac(R)$ such that for some free resolution $F$ of $M$ we have $F_{\ge i}=C_{\ge i}$ for some $i\ge 0$. In this case, $C$ is contractible if and only if $M$ has finite projective dimension over $R$.
\end{chunk}

\begin{proposition}\label{contractminimal} Assume that $(R,\m,k)$ is local and let $C\in\Ktac(R)$.  Then $C$ is contractible at $n$ for some $n\in\mathbb Z$ if and only if $C$ is contractible.
\end{proposition}

\begin{proof} The `if' direction is trivial.

For the `only if' direction, let $C\in\Ktac(R)$.  Then by \ref{directsum} we can assume that $C=E\oplus X$ where
$E$ is contractible and $X$ is minimal. Suppose that $C$ is contractible at $n$. Then there exist maps $s_n:C_n\to C_{n+1}$ and $s_{n-1}:C_{n-1}\to C_n$ such that $\partial^C_{n+1}s_n+s_{n-1}\partial^C_n$ is an isomorphism.
It follows that $\partial^C_{n+1}s_n+s_{n-1}\partial^C_n\otimes_Rk$ is an isomorphism.  Taking into account the minimality of $X$, we see that $\rank(\partial^C_{n+1}s_n+s_{n-1}\partial^C_n\otimes_Rk)$ is at most 
$\rank E_n$, which is a contradiction unless $X_n=0$. Suppose that $h=\inf\{i>n\mid X_i\ne 0\}$ is finite.
Then exactness of $X$ at $h$ and minimality implies $X_h=\m X_h$. Therefore Nakayama's Lemma implies 
that $X_h=0$, contradiction. Similarly, if $l=\sup\{i<n\mid X_i\ne 0\}$ is finite then exactness of
$\Hom_R(X,R)$ at $\Hom_R(X_l,R)$ and minimality implies that $\Hom_R(X_l,R)=\m \Hom_R(X_l,R)$, which implies 
by Nakayama's Lemma that $\Hom_R(X_l,R)=0$, which only happens if $X_l=0$, a contradiction.  Thus $X$ must be the zero complex, and therefore we have $C=E$ is contractible.
\end{proof}


\section{The generic hypersurface and rank varieties}\label{sec:3}

Let $Q$ be a regular local ring with maximal ideal $\n$ and residue
field $k$. Let $f_1,\dots,f_c$ be a regular sequence in $Q$ with $c\ge 2$.
Throughout the remainder of this paper we let $P$ denote the polynomial ring $Q[x_1,\dots,x_c]$
in the commuting indeterminates $x_1,\dots,x_c$. Setting the degree of each $x_i$ to be 1, we realize $P$ as a standard graded ring with $Q$ concentrated in degree 0. Notice that since $Q$ is a unique factorization domain, so is $P$. Thoughout the paper we consider the generic hypersurface $R$, which is defined to be the quotient of $P$
\[
R = P/(w)
\]
by the principal ideal generated by $w=f_1x_1+\cdots +f_cx_c$.  We claim that $w$ is irreducible in 
$P$, for if $w=ab$ for $a,b\in P$, then by degree considerations we must have, without loss of generality, 
$a\in Q$ and $b=b_1x_1+\cdots+b_cx_c$ with $b_i\in Q$ for $i=1,\dots,c$.  Equating coefficients yields 
$f_1=ab_1$ and $f_2=ab_2$. Thus $b_1f_2=b_1ab_2\in(f_1)$. Since $f_1,f_2$ form a regular sequence, it follows that $b_1\in(f_1)$, which forces $a$ to be a unit.  Since $P$ is a unique factorization domain, it follows that the principal ideal $(w)$ is prime, and so $R$ is an integral domain.

Let $C$ be a complex of finitely generated free $R$-modules.  Since $R$ is an integral domain, 
by \ref{maprank}, each map $\partial_n^C$ has a well-defined rank, $\rank\partial_n^C$.

\begin{chunk}\label{rankequality} Suppose that $C$ is an exact complex of finitely generated free $R$-modules. Then for all $n\in\mathbb Z$ we have the equality
\[
\rank\partial^C_n+\rank\partial^C_{n+1}=\rank C_n
\]
Indeed, this follows immediately from the definition of $\rank\partial^C_n$, \ref{rank=} and the fact that $S^{-1}C$ is an exact sequence of vector spaces. 
\end{chunk}

\subsection*{Graded $\Ktac(R)$} Since $R$ is a graded ring, we consider the category
$\grKtac(R)$ of graded totally acyclic complexes over $R$.  The objects are those $C\in \Ktac(R)$ such that each doug $\partial_n^C$ is homogeneous and the morphisms are the homogenous homotopy equivalence classes of homogeneous chain maps. Note that in this case the ideals 
$I_r(\partial^C_n)$ of $R$ are homogeneous.

\begin{chunk}\label{grktac} If $M$ is a finitely generated graded $R$-module then $M$ has a graded complete resolution 
$C\in\grKtac(R)$. Indeed, the proof is a graded analogue of the construction of Shamash 
\cite{Shamash}, see also \cite{E}.
\end{chunk}

\subsection*{Rank varieties}
For each $\alpha=(\alpha_1,\dots,\alpha_c)\in\mathbb P^{c-1}_k$ we fix preimages $a_i\in Q$ of 
$\alpha_i\in k$ for $i=1,\dots,c$ and let $R_\alpha$ denote the quotient
\[
R_\alpha = R/(x_1-a_1,\dots,x_c-a_c).
\]
Of course $R_\alpha$ depends on the choice of preimages, but as we shall see, the rank variety we consider below is independent of this choice.
We note that $R_\alpha$ is isomorphic to the specialized local hypersurface,
\[
R_\alpha\cong Q/(a_1f_1+\cdots+a_cf_c).
\]
For $C\in\grKtac(R)$ and $\alpha\in\mathbb P^{c-1}_k$ we let $C_\alpha$ denote the complex
\[
C_\alpha=C\otimes_R R/(x_1-a_1,\dots,x_c-a_c)
\]
which is an object of $\Ktac(R_\alpha)$. Indeed, according to \cite[pg. 397]{J}, $x_1-a_1,\dots,x_c-a_c$ is a regular sequence in $R$, and so $R/(x_1-a_1,\dots,x_c-a_c)$ has projective dimension $c$ over 
$R$. Then for arbitrary $n$, $\Tor^R_{c+1}(R/(x_1-a_1,\dots,x_c-a_c),\coker\partial^C_{n-c})=0$ shows that $C_\alpha$ is exact at degree $n$. Thus $C_\alpha$ is exact, and therefore totally acyclic over 
$R_\alpha$ \cite[lem. 2.4]{AM}.

\begin{definition} Let $C\in\grKtac(R)$.  We define the \emph{rank variety} of $C$ to be
\[
\V(C)=\{\alpha\in\mathbb P^{c-1}_k\mid C_\alpha\text{ is not contractible}\}
\]
\end{definition}

Note that we are now able to assign an invariant of dimension to objects in $\grKtac(R)$.
We first want to show that the rank variety is independent of the choice of preimages 
$a_1,\dots,a_c$ for $\alpha\in\mathbb P^{c-1}_k$. 

\begin{theorem}\label{independent} Let $C\in\grKtac(R)$ and $a_1,\dots,a_c$ and $a'_1,\dots,a'_c$ be sequences of elements of $Q$ such that $a_i-a'_i\in\n$ for $i=1,\dots,c$. Then
\[
C\otimes_RR/(x_1-a_1,\dots,x_c-a_c)
\] 
is contractible if and only if 
\[
C\otimes_RR/(x_1-a'_1,\dots,x_c-a'_c)
\] 
is contractible 
\end{theorem}

\begin{proof} Assume that $C'=C\otimes_RR/(x_1-a_1,\dots,x_c-a_c)$ is contractible.  Denote by
$\partial'$ the differentials $\partial^C\otimes_RR/(x_1-a_1,\dots,x_c-a_c)$. By assumption
there exist maps $s'_n:C'_n\to C'_{n+1}$ such that  
$\partial'_{n+1}s'_n+s'_{n-1}\partial'_n$ is an isomorphism for all $n$.  Lift the maps $s'_n$ to maps 
$S_n:C_n\to C_{n+1}$ and set $s''_n=S_n\otimes_RR/(x_1-a'_1,\dots,x_c-a'_c)$. Note that 
the rings $R\otimes_RR/(x_1-a_1,\dots,x_c-a_c)$ and $R\otimes_RR/(x_1-a'_1,\dots,x_c-a'_c)$
are both local with common residue field $k$.  Letting $\partial''$ denote the differentials of
$C''=C\otimes_RR/(x_1-a'_1,\dots,x_c-a'_c)$, we see that the maps 
$(\partial'_{n+1}s'_n+s'_{n-1}\partial'_n)\otimes_{R'}k$ and $(\partial''_{n+1}s''_n+s''_{n-1}\partial''_n)\otimes_{R''}k$ agree, and thus, by \ref{surjective1}, $\partial'_{n+1}s'_n+s'_{n-1}\partial'_n$ 
and $\partial''_{n+1}s''_n+s''_{n-1}\partial''_n$ are isomorphisms simultaneously. Thus $C''$ is also contractible.
\end{proof}

There is another issue of well-definedness that we now address. 

\begin{proposition}\label{homeq}
Suppose that $C$ and $D$ in $\grKtac(R)$ are homotopically equivalent, that is, isomorphic in 
$\grKtac(R)$, then $\V(C)=\V(D)$.
\end{proposition}

\begin{proof} First note that if a morphism $f:C\to D$ in $\grKtac(R)$ is null-homotopic, then for any $\alpha$ the induced morphism $f_\alpha:C_\alpha\to D_\alpha$ in $\Ktac(R_\alpha)$ is null-homotopic. It follows that if $C$ and $D$ are homotopically equivalent then so are 
$C_\alpha$ and $D_\alpha$ for any $\alpha$. Therefore by \ref{contract equiv}, $C_\alpha$ is contractible if and only if $D_\alpha$ is contractible. It follows that $\V(C)=\V(D)$.
\end{proof}

We now turn our attention to showing that the rank variety is Zariski closed.  For an ideal $I$ 
in the polynomial ring $k[x_1,\dots,x_c]$ we let $\Z(I)$ denote the zero set
\[
\Z(I)=\{\alpha\in\mathbb P_k^{c-1}\mid p(\alpha)=0\text { for all } p\in I\}
\]
Let $r_n=\rank\partial^C_n$ and $\overline{I_{r_n}(\partial^C_n)}$ denote the image of the ideal 
$I_{r_n}(\partial^C_n)$ under the natural surjection $R\to k[x_1,\dots,x_c]$. Note that since 
$C\in\grKtac(R)$, the ideals $I_{r_n}(\partial^C_n)$ are homogeneous ideals of $R$.  Thus the ideals 
$\overline{I_{r_n}(\partial^C_n)}$ are homogeneous ideals of $k[x_1,\dots,x_c]$.

\begin{theorem}\label{closed} Let $C\in\grKtac(R)$. For any $n\in\mathbb Z$ we have the equality of sets
\[
\V(C)=\Z\left(\overline{I_{r_n}(\partial^C_n)}\cap\overline{I_{r_{n+1}}(\partial^C_{n+1})}\right)
\]
In particular, the rank variety $\V(C)$ of $C$ is a Zariski closed set.
\end{theorem}

\begin{proof} Let $n\in\mathbb Z$ and suppose $\alpha$ is not in $\V(C)$, that is, $C_\alpha$ is contractible. Then there exist maps
$s_n:(C_\alpha)_n\to(C_\alpha)_{n+1}$ and $s_{n-1}:(C_\alpha)_{n-1}\to(C_\alpha)_n$
such that $\partial_{n+1}^{C_\alpha}s_n+s_{n-1}\partial_n^{C_\alpha}$ is an isomorphism. 
It follows from \ref{rank=} that the complex of $k$-vector spaces $C_\alpha\otimes_Qk$ is exact at $n$.  Therefore by \ref{rank=} $\rank(\partial^{C_\alpha}_{n+1}\otimes k)+\rank(\partial^{C_\alpha}_n\otimes k)=\rank((C_\alpha)_n\otimes_Qk)$.  Since $\rank((C_\alpha)_n\otimes_Rk)=\rank C_n$, it follows from the previous equality that 
$\rank(\partial^{C_\alpha}_{n+1}\otimes k)=r_{n+1}$ and $\rank(\partial^{C_\alpha}_n\otimes k)=r_n$ Thus 
$\alpha\notin \Z(\overline{I_{r_n}(\partial^C_n)})$ and 
$\alpha\notin \Z(\overline{I_{r_{n+1}}(\partial^C_{n+1})})$. In other words, 
$\alpha\notin \Z\left(\overline{I_{r_n}(\partial^C_n)}\cap\overline{I_{r_{n+1}}(\partial^C_{n+1})}\right)$.

For the converse, assume $\alpha$ is not in $\Z\left(\overline{I_{r_n}(\partial^C_n)}\cap\overline{I_{r_{n+1}}(\partial^C_{n+1})}\right)$ for some $n\in\mathbb Z$. Therefore 
$\rank(\partial^{C_\alpha}_{n+1}\otimes k)=r_{n+1}$ and $\rank(\partial^{C_\alpha}_n\otimes k)=r_n$ 
and it follows that 
$\rank(\partial^{C_\alpha}_{n+1}\otimes k)+\rank(\partial^{C_\alpha}_n\otimes k)=\rank((C_\alpha)_n\otimes_Qk)$.  Thus by \ref{rank=}, $C_\alpha\otimes k$ is exact at $n$ and there exist maps $s_n:(C_\alpha)_n\otimes k\to(C_\alpha)_{n+1}\otimes k$ and $s_{n-1}:(C_\alpha)_{n-1}\otimes k\to(C_\alpha)_n\otimes k$
such that $(\partial_{n+1}^{C_\alpha}\otimes k)s_n+s_{n-1}(\partial_n^{C_\alpha}\otimes k)$ is an isomorphism. We lift $s_n$ and $s_{n-1}$ to maps
$s^\sharp_n:(C_\alpha)_n\to(C_\alpha)_{n+1}$ and $s^\sharp_{n-1}:(C_\alpha)_{n-1}\to(C_\alpha)_n$
such that $s^\sharp_n\otimes k=s_n$ and $s^\sharp_{n-1}\otimes k=s_{n-1}$. Then by \ref{surjective1} the map
$\partial_{n+1}^{C_\alpha}s^\sharp_n+s^\sharp_{n-1}\partial_n^{C_\alpha}$ is an isomorphism and so a contraction at $n$. Finally, by \ref{contractminimal} we see that $C_\alpha$ is contractible and thus 
$\alpha\notin\V(C)$.
\end{proof}

\begin{corollary}\label{closedcor} Let $C\in\grKtac(R)$ and $\alpha\in\mathbb P^{c-1}_k$.  Then $C_\alpha$ is 
contractible if and only if for all $n\in\mathbb Z$
\[
\rank(\partial^{C_\alpha}_n\otimes_R k)+\rank(\partial^{C_\alpha}_{n+1}\otimes_Rk)=\rank C_n
\]
\end{corollary}

\begin{proof} The result follows from the previous theorem since 
$\alpha\notin\Z\left(\overline{I_{r_n}(\partial^C_n)}\right)$ if and only if
$\rank(\partial^{C_\alpha}_n\otimes_Rk)=\rank\partial^C_n$. Then we use \ref{rankequality}.
\end{proof}

We list some basic properties of rank varieties.

\begin{theorem}\label{properties} For complexes in $\grKtac(R)$ the following hold.
\begin{enumerate}
\item $\V(C)=\V(\Sigma C)$.
\item $\V(C\oplus C')=\V(C)\cup\V(C')$.
\item $\V(C)=\V(\Hom_{\grKtac}(C,R))$.
\item If $C_1\to C_2\to C_3\to \Sigma C_1$ is an exact triangle in $\grKtac(R)$ then
$\V(C_i)\subseteq \V(C_j)\cup\V(C_k)$ for $\{i,j,k\}=\{1,2,3\}$.
\end{enumerate}
\end{theorem}

\begin{proof} Property (1) is clear from the definition. Property (2) follows since $(C\oplus C')_\alpha$ is contractible if and only if both $C_\alpha$ and $C'_\alpha$ are contractible.

Let $(\quad)^*$ denote $\Hom_{\grKtac}(\quad,R)$. For (3) we use \ref{closedcor} and the facts that 
$\rank((\partial^{C_\alpha}_n)^*\otimes_Rk)=\rank(\partial^{C_\alpha}_n\otimes_Rk)$ and 
$\rank C_n^*=\rank C_n$.

For (4) it suffices to prove $\V(C_3)\subseteq\V(C_1)\cup\V(C_2)$ since this, rotation and (1) give the other containments.  

Up to isomorphism of triangles we have $C_3=\Cone(f)$ for some morphism $f:C_1\to C_2$. Thus we may assume that $\partial_n^{C_3}$ has the form $\left(\begin{array}{cc} \partial_n^{C_2} & f_{n-1}\\ 0 & -\partial_{n-1}^{C_1}\end{array}\right)$. It follows that $\rank\partial_n^{C_3}\ge\rank\partial_{n-1}^{C_1}+\rank\partial_n^{C_2}$. Thus if both $C_1$ and $C_2$ are contractible, then so is $C_3$. This is another way to state the result.
\end{proof}

Another important result for support or rank varieties via restriction of scalars is Dade's Lemma, which states that the support or rank variety is trivial if and only if the object is trivial, for example, contractible. We don't quite get the same result for rank varieties via extension of scalars.

\begin{proposition}\label{Dades} For $C\in\grKtac(R)$ we have $\V(C)=\varnothing$ if and only if 
$C$ is contractible or $\sqrt{\overline{I_{r_n}(\partial_n^C)}}=(x_1,\dots,x_c)$ for all $n$ (equivalently, for two consecutive valus of $n$).
\end{proposition}

\begin{proof} Suppose that $\V(C)=\varnothing$ and that $C$ is not contractible. Then $C_\alpha$ is contractible for all $\alpha\in\mathbb P^{c-1}_k$. Therefore $C_\alpha\otimes_Rk$ is contractible for all $\alpha$ and so by \ref{rank=} we have $\rank(\partial_n^{C_\alpha}\otimes k)=r_n$ for all $n$ and all $\alpha$. It follows that the ideal $\overline{I_{r_n}(\partial_n^C)}$ of $k[x_1,\dots,x_c]$ has no zeros in $\mathbb P^{c-1}_k$; in other words $\sqrt{\overline{I_{r_n}(\partial_n^C)}}=(x_1,\dots,x_c)$. For the converse, clearly if $C$ is contractible then so is $C_\alpha$ for all $\alpha$. Therefore suppose $\sqrt{\overline{I_{r_n}(\partial_n^C)}}=(x_1,\dots,x_c)$ for two consecutive values of $n$. Then \ref{closed} shows that $\V(C)=\varnothing$.
\end{proof}

Of course there exist contractible $C\in\grKtac(R)$, for example, $0\to R\xrightarrow{1} R \to 0$. We show by example that the second condition of \ref{Dades} does in fact occur.

\begin{example} Let $Q=k\llbracket x,y\rrbracket$, where $k$ is a field, and $f_1=x^2$, $f_2=y^2$.  Then the generic hypersurface is $R=Q[x_1,x_2]/(x^2x_1+y^2x_2)$. The complex 
\[
C: \cdots\to R^2 \xrightarrow{\left(\begin{array}{cc} x^2 & y^2\\ -x_2 & x_1 \end{array}\right)} R^2
\xrightarrow{\left(\begin{array}{cc} x_1 & -y^2\\ x_2 & x^2 \end{array}\right)} R^2 \to\cdots
\]
is a complete resolution of $Q[x_1,x_2]/(f_1,f_2)$ over $R$. By inspection, each map $\partial^C_n$ has rank one and $\overline{I_1(\partial_n^C)}=(x_1,x_2)$.
\end{example}

\section{Realizability}

In this section we retain that $R$ denotes the generic hypersurface and show that any projective variety in 
$\mathbb P^{c-1}_k$ can be realized as the rank variety of some $C\in\Ktac R$.  

Let $C\in\Ktac(R)$ and 
$p\in R$.  Then multiplication by $p$ defines a chain endomorphism $\mu_p:C\to C$.  We denote the mapping cone of this endomorphism by $C^p$.

\begin{lemma}\label{lemma} Let $C\in\Ktac(R)$ and $p\in R$.  We let $\overline p$ denote the image of $p$ under the natural surjection $R\to k[x_1,\dots,x_c]$. Then
\[
\V(C^p)=\V(C)\cap\Z(\overline p)
\]
\end{lemma} 

\begin{proof} The statement is obvious if $\overline p=0$ so we assume $\overline p\ne 0$. Recall that $\partial_n^{C^p}$ can be represented by the matrix
$\left(\begin{array}{cc} \partial_n^C & p\\ 0 & -\partial_{n-1}^C \end{array}\right)$
where $p$ represents the $\rank C_{n-1}\times\rank C_{n-1}$ scalar matrix of $p$. From this it
is to easy to see that $\rank\partial^{C^p}_n=\rank C_{n-1}$.

We first show the reverse containment. Suppose that $\alpha\notin \V(C^p)$ and that 
$\alpha\in\Z(\overline p)$. This means that $C^p_\alpha$ is contractible and 
$\overline p(\alpha)=0$. Therefore by \ref{closedcor}, for all $n\in\mathbb Z$ we have 
$\rank(\partial^{C^p_\alpha}_{n+1}\otimes_Rk)+\rank(\partial^{C^p_\alpha}_n\otimes_Rk)=
\rank C^p_n=\rank C_n+\rank C_{n-1}$. Since $\rank(\partial^{C_\alpha}_n\otimes_Rk)+
\rank(\partial^{C_\alpha}_{n-1}\otimes_Rk)=\rank(\partial^{C^p_\alpha}_n\otimes_Rk)\le\rank C_{n-1}$ and $\rank(\partial^{C_\alpha}_{n+1}\otimes_Rk)+
\rank(\partial^{C_\alpha}_n\otimes_Rk)=\rank(\partial^{C^p_\alpha}_{n+1}\otimes_Rk)\le\rank C_n$, it follows that  
$\rank(\partial^{C_\alpha}_{n+1}\otimes_Rk)+
\rank(\partial^{C_\alpha}_n\otimes_Rk)=\rank C_n$ for all $n$, and so by \ref{closedcor} $C_\alpha$ is contractible. In other words, $\alpha\notin\V(C)$.

Now suppose that $\alpha\notin \V(C^p)$ and that $\alpha\in\V(C)$. Therefore for all $n\in\mathbb Z$ we have $\rank(\partial^{C^p_\alpha}_{n+1}\otimes_Rk)+\rank(\partial^{C^p_\alpha}_n\otimes_Rk)=
\rank C^p_n=\rank C_n+\rank C_{n-1}$ and for some $n\in\mathbb Z$, 
$\rank(\partial^{C_\alpha}_n\otimes_R k)+\rank(\partial^{C_\alpha}_{n+1}\otimes_Rk)<\rank C_n$.
It follows that $\overline p(\alpha)\ne 0$ and so $\alpha\notin\Z(\overline p)$.

For the forward containment, assume that $\alpha\notin\V(C)\cap\Z(\overline p)$.  Thus either
$\alpha\notin\V(C)$ or $\alpha\notin\Z(\overline p)$. In either case we have that 
$\rank(\partial^{C^p_\alpha}_n\otimes_Rk)+\rank(\partial^{C^p_\alpha}_{n+1}\otimes_Rk)=\rank C^p_n$ for all $n\in\mathbb Z$. Thus by \ref{closedcor}, $C^p_n$ is contractible, that is, 
$\alpha\notin\V(C^p)$.
\end{proof}

\begin{proposition}\label{K} Let $K\in\grKtac(R)$ be a graded complete resolution of $k$.  Then
\[
\V(K)=\mathbb P^{c-1}_k
\]
\end{proposition}

\begin{proof} Note that for every $\alpha\in\mathbb P^{c-1}_k$, $R_\alpha$ is a singular hypersurface and so 
$k$ has a infinite projective dimension over $R_\alpha$.  Thus by \ref{CR} $K_\alpha$ is not contractible for every $\alpha$.  The result follows.
\end{proof}

\begin{theorem}\label{realize} Let $W\subseteq \mathbb P^{c-1}_k$ be a projective variety.  Then there exist 
$C\in\Ktac(R)$ such that $\V(C)=W$.
\end{theorem}

\begin{proof} Let $p_1,\dots,p_m\in R$, and $\overline p_1,\dots,\overline p_m$ their images in 
$k[x_1,\dots,x_c]$, be such that $W=\Z((\overline p_1,\dots,\overline p_m))$. Let $K\in\Ktac(R)$ be a complete resolution of $k$. We construct $C\in\Ktac(R)$ by induction on $m$.

For $m=1$, let $p=p_1$ and consider $C_{(1)}=K^p$. Then by \ref{lemma} we have
$\V(C_{(1)})=\V(K)\cap\Z(\overline p)=\mathbb P^{c-1}_k\cap\Z(\overline p)=\Z(\overline p)$ and we are done in this case.

Suppose that $C_{(m-1)}\in\Ktac(R)$ has been constructed such that 
$\V(C_{(m-1)})=\Z((\overline p_1,\dots,\overline p_{m-1}))$  Now define $C=C_{(m-1)}^{p_m}$. Then by 
\ref{lemma} we have
$\V(C)=\V(C_{(m-1)})\cap\Z(\overline p_m)=\Z((\overline p_1,\dots,\overline p_m))$ and induction is complete.
\end{proof}

\section{Failure of Decomposability}

A result that holds for rank and support varieties over group algebras and complete intersections is that if a module is indecomposable then its support variety is connected.  We show by example that this property fails for rank varieties of modules over the generic hypersurface.

\begin{example} Consider the regular local ring $Q=k\llbracket x,y\rrbracket$, where $k$ is a field, and $Q$-regular sequence 
$f_1=x^2,f_2=y^2$. Then the corresponding generic hypersurface is
\[
R=Q[x_1,x_2]/(x^2x_1+y^2x_2)
\]
Let $K$ be a complete resolution of $k$ over $R$ and consider $x_1x_2\in R$.  Then the complex $K^{x_1x_2}$ is
given by
\[
K^{x_1x_2}: \cdots \to R^4 \xrightarrow{\partial} R^4\xrightarrow{\partial'} R^4 \xrightarrow{\partial} R^4\to\cdots
\]
where the dougs $\partial$ and $\partial'$ are given by the matrices
$D=\left(\begin{array}{cc|cc} -y & xx_1 & x_1x_2 & 0\\
																													x & yx_2 & 0 & x_1x_2\\ \hline
																													0 & 0 & yx_2 & -xx_1\\
																													0 & 0 & -x & -y\end{array}\right)$
and																																																		
$D'=\left(\begin{array}{cc|cc} -yx_2 & xx_1 & x_1x_2 & 0\\
																													x & y & 0 & x_1x_2\\ \hline
																													0 & 0 & y & -xx_1\\ 
																													0 & 0 & -x & -yx_2\end{array}\right)$,
																													respectively, with respect to the standard basis of $R^4$.
From \ref{lemma} and \ref{K} we have that $\V(K^{x_1x_2})=\Z(\overline{x_1x_2})$, which is obviously disconnected.
We will show that $K^{x_1x_2}$ is indecomposable. To this end we show that $M=\coker\partial$ is indecomposable. Suppose that $\epsilon:M\to M$ is an idempotent map.  We lift this to a commutative diagram
\[
\xymatrixrowsep{2pc}
\xymatrixcolsep{3pc}
\xymatrix{
R^4 \ar@{->}[r]^{\partial}\ar@{->}[d]^{\beta} & R^4 \ar@{->}[r]\ar@{->}[d]^{\alpha} & M \ar@{->}[r]\ar@{->}[d]^{\epsilon} & 0 \\
R^4 \ar@{->}[r]^{\partial} & R^4 \ar@{->}[r] & M \ar@{->}[r] & 0
}
\] 
Let $(a_{ij})$ and $(b_{ij})$ be matrices representing $\alpha$ and $\beta$, respectively, with respect to the standard basis of $R^4$. Then we have the matrix equation $(a_{ij})D=D(b_{ij})$. Working modulo $\m^2R+\m(x_1,x_2)^2+(x_1^2,x_2^2)$ we have the matrix equation
$(\overline{a_{ij}})\overline D=\overline D(\overline{b_{ij}})$ and can assume the entries
$\overline{a_{ij}}$ and $\overline{b_{ij}}$ belong to $k+kx_1+kx_2$. One sees right away that
$\overline a_{31}=\overline a_{32}=\overline a_{41}=\overline a_{42}=0$ and 
$\overline b_{31}=\overline b_{32}=\overline b_{41}=\overline b_{42}=0$. Further, we have
$\overline a_{11}=\overline b_{11}=\overline a_{22}=\overline b_{22}=\overline a_{33}=\overline b_{33}=\overline a_{44}=\overline b_{44}$. Looking at the $22$-entry of the matrix equation, if 
$\overline a_{21}\ne 0$, then $\overline b_{12}=\overline a_{21}x_1$, but this contradicts the $12$-entry of the matrix equation. Thus $\overline a_{21}=0$.  Similarly, looking at the $43$-entry of the matrix equation above, if $\overline a_{43}\ne 0$ then $\overline b_{43}=-\overline a_{43}x_2$. Now the $33$-entry implies that $\overline a_{34}=0$ and the $34$-entry implies that 
$\overline b_{43}=0$. Finally, the $44$-entry then implies that $\overline a_{43}=0$, a contradiction.  

We have thus shown that $(\overline{a_{ij}})$ is upper triangular.  If $\overline{a_{11}}$ is nonzero in $k$ then $(\overline{a_{ij}})$ is surjective, which implies $\overline\epsilon$ is surjective.  By Nakayama's Lemma, the same is true of $\epsilon$, which implies that $\epsilon$ is the unit idempotent. If $\overline{a_{11}}\in kx_1+kx_2$ then $(\overline{a_{ij}})$ is nilpotent, which implies the same for $\overline\epsilon$.  In this case it follows that $\epsilon$ is the zero idempotent.  Thus the only idempotent maps $M\to M$ are the unit or zero, and so $M$ is indecomposable.
\end{example}

\section{Modules}\label{modules}
In this section we define rank varieties for modules over the generic hpersurface
$R=P/(w)$, where $w=f_1x_1+\cdots +f_cx_c$, and use the results of the previous section to establish the same results for graded $R$-modules instead of graded totally acyclic $R$-complexes.

Given $\alpha=(\alpha_1,\dots,\alpha_c)\in\mathbb P^{c-1}_k$ we fix preimages $a_i\in Q$ of
$\alpha_i\in k$ and again define the local ring $R_\alpha$ to be
\[
R_\alpha=R/(x_1-a_1,\dots,x_c-a_c)
\]
For an $R$-module $M$ we let $M_\alpha$ denote the $R_\alpha$-module
\[
M_\alpha=M/(x_1-a_1,\dots,x_c-a_c)M
\]

\begin{definition} Let $M$ be a finitely generated graded $R$-module. We define the \emph{rank variety}
of $M$ to be
\[
\V(M)=\{\alpha\in\mathbb P^{c-1}_k\mid M_\alpha\text{ has infinite projective dimension over 
} R_\alpha\}
\]
\end{definition}

We show that the rank variety is independent of the choice of preimages 
$a_1,\dots,a_c$ for $\alpha\in\mathbb P^{c-1}_k$. 

\begin{theorem} Let $M$ be a finitely generated graded $R$-module and $a_1,\dots,a_c$ and $a'_1,\dots,a'_c$ be sequences of elements of $Q$ such that $a_i-a'_i\in\n$ for $i=1,\dots,c$. Then
\[
M/(x_1-a_1,\dots,x_c-a_c)M
\] 
has finite projective dimension over $R/(x_1-a_1,\dots,x_c-a_c)$ if and only if 
\[
M/(x_1-a'_1,\dots,x_c-a'_c)M
\] 
has finite projective dimension over $R/(x_1-a'_1,\dots,x_c-a'_c)$ 
\end{theorem}

\begin{proof} Let $C\in\grKtac(R)$ be a complete resolution of $M$ (see \ref{grktac}). Then by 
\ref{independent},  
$C\otimes_RR/(x_1-a_1,\dots,x_c-a_c)$ is contractible if and only if 
$C\otimes_RR/(x_1-a'_1,\dots,x_c-a'_c)$ is contractible. Therefore by \ref{CR} $M/(x_1-a_1,\dots,x_c-a_c)M$ 
has finite projective dimension over $R/(x_1-a_1,\dots,x_c-a_c)$ if and only if 
$M/(x_1-a'_1,\dots,x_c-a'_c)M$ 
has finite projective dimension over $R/(x_1-a'_1,\dots,x_c-a'_c)$.
\end{proof}

\begin{chunk}\label{varieties=}
Let $M$ be a finitely generated graded $R$-module, and $C\in\Ktac(R)$ a graded complete resolution of $M$.  Then \ref{CR} gives the equality of rank varieties $\V(M)=\V(C)$.
\end{chunk}

An immediate consequence of \ref{varieties=} is that rank varieties for modules are also Zariski closed. 

\begin{theorem}\label{closedM} Let $M$ be a finitely generated graded $R$-module. Then the rank variety 
$\V(M)\subseteq\mathbb P^{c-1}_k$ of $M$ is Zariski closed.
\end{theorem}

The realizability question for modules is easy, given that it holds for totally acyclic complexes.

\begin{theorem}\label{realizeM} Let $W\subseteq \mathbb P^{c-1}_k$ be a projective variety.  Then there exists a 
finitely generated graded $R$-module $M$ such that $\V(M)=W$.
\end{theorem}

\begin{proof} By \ref{realize} we have $C\in\grKtac(R)$ such that $\V(C)=W$.  
Define $M=\Image\partial^C_0$. Then $C$ is a complete resolution of $M$ and then 
\ref{varieties=} gives the result.
\end{proof}

\subsection{Affine Patches} In this subsection we consider the realizability question for modules over the associated complete intersection ring $S=Q/(f_1,\dots,f_c)$. Although it is known that every projective variety in $\mathbb P^{c-1}_k$ is the support variety (equivalently, rank variety) of some finitely generated $S$-module \cite{Bergh}, see also \cite{AJ} and \cite{AI}, using the results of this paper we can realize the affine patches 
$\mathbb A^i_k=\{(\alpha_1,\dots,\alpha_c)\in\mathbb P^{c-1}_k\mid \alpha_i=1\}$
of a projective variety $W\subseteq\mathbb P^{c-1}_k$ in a relatively painless way, albeit by non-finitely generated $S$-modules. For existing proofs, it is difficult to ascertain a presentation of the module realizing the variety.

\begin{theorem} Let $W\subseteq\mathbb P^{c-1}_k$ be a projective variety. Then there exist (not
necessarily finitely generated) graded $S$-modules $M^i$ such that 
\[
\V(M^i)\cap\mathbb A^i_k=W\cap\mathbb A^i_k
\]
for $i=1,\dots,c$.
\end{theorem}

\begin{proof} First notice that for each $i=1,\dots,c$ the sequence 
\[
s_i=f_1\dots,f_{i-1},x_i-1,f_{i+1},\dots,f_c
\] 
is a regular on $R$. By \ref{realizeM} there exists a finitely generated graded $R$-module $M$ such that $\V(M)=W$. Since $M$ is the zeroth image in a 
totally acyclic complex, the sequences $s_i$ are also regular on $M$. Now for each 
$i=1,\dots,c$ define $M^i=M/(s_i)M$.  Then each $M^i$ is an $S$-module and finitely 
generated over $R/(s_i)\cong S[x_1,\dots,\widehat x_i,\dots,x_c]$. 

Note that for $\alpha\in\mathbb A^i_k$ the sequence $f_1\dots,\widehat f_i,\dots,f_c$ is regular on 
$M_\alpha=M/(x_1-\alpha_1,\dots,x_i-1,\dots,x_c-\alpha_c)M$ and $M^i_\alpha=M^i/(x_1-\alpha_1,\dots,x_i-1,\dots,x_c-\alpha_c)M^i$ is isomorphic to $M_\alpha/(f_1\dots,\widehat f_i,\dots,f_c)M_\alpha$.  Thus $M^i_\alpha$ and $M_\alpha$ have finite projective dimension over $R_\alpha$ simultaneously.  This establishes the claimed equality of varieties. 
\end{proof}

\end{document}